\documentclass[12pt]{amsart}

\usepackage{latexsym, amsmath, amscd, amssymb, amsthm} 
\usepackage[T2A]{fontenc}
 \usepackage[cp1251]{inputenc}
\usepackage[ukrainian, russian, english]{babel}
\newtheorem{te}{Theorem }

 \newtheorem{lm}{Lemma}

\begin{document}
\noindent

 \title{ The Poincar\'e series of the joint invariants and covariants of the two binary forms}

\author{L. Bedratyuk}\address{Khmelnitskiy national university, Insituts'ka, 11,  Khmelnitskiy, 29016, Ukraine}

\begin{abstract} Let $\mathcal{I}_{d_1,d_2}$ and $\mathcal{C}_{d_1,d_2}$ be  the algebras  of joint invariants and joint covariants of the two binary forms of degrees $d_1$ and $d_2.$  Formulas for computation of the Poincar\'e series $\mathcal{PI}_{d_1,d_2}(z),$ $ \mathcal{PC}_{d_1,d_2}(z)$  of the algebras  is found. 
By using these formulas, we have computed the series  for  $d_1,d_2 \leq 20.$
\end{abstract}
\maketitle

\noindent
{\bf 1.} Let   $V_{d_1}, V_{d_2}$ be the complex vector spaces of the binary forms of degrees   $d_1$ and   $d_2$ endowed with the natural action of the complex group $SL_2.$ Consider the induced action of the group  on  the algebras of the polynomial functions  $\mathcal{O}(V_{d_1}\oplus V_{d_2})$  and  $\mathcal{O}(V_{d_1}\oplus V_{d_2}\oplus \mathbb{C}^2).$ The algebras   $$\mathcal{I}_{d_1,d_2}:=\mathcal{O}(V_{d_1}\oplus V_{d_2})^{SL_2} \text{  and  }  \mathcal{C}_{d_1,d_2}:=\mathcal{O}(V_{d_1}\oplus V_{d_2}\oplus \mathbb{C}^2)^{SL_2}$$   is called the algebra of joint invariants and the algebra of joint covariants for the binary forms. 
The reductivity of  $SL_2$  implies that the algebras  $\mathcal{I}_{d_1,d_2},$  $\mathcal{C}_{d_1,d_2}$  are  finite generated graded algebras
$$
\begin{array}{l}
\mathcal{I}_{d_1,d_2}=(\mathcal{I}_{d_1,d_2})_0+(\mathcal{I}_{d_1,d_2})_1+\cdots+(\mathcal{I}_{d_1,d_2})_i+ \cdots,\\
\\
\mathcal{C}_{d_1,d_2}=(\mathcal{C}_{d_1,d_2})_0+(\mathcal{C}_{d_1,d_2})_1+\cdots+(\mathcal{C}_{d_1,d_2})_i+ \cdots,
\end{array}
$$
and the vector spaces  $(\mathcal{I}_{d_1,d_2})_i,$ $(\mathcal{C}_{d_1,d_2})_i$ are the finite dimensional. The formal power series  $\mathcal{PI}_{d_1,d_2}(z), \mathcal{CI}_{d_1,d_2}(z)  \in \mathbb{Z}[[z]],$
$$
\mathcal{PI}_{d_1,d_2}(z)=\sum_{i=0}^{\infty }\dim((\mathcal{I}_{d_1,d_2})_i) z^i,  \mathcal{CI}_{d_1,d_2}(z) =\sum_{i=0}^{\infty }\dim((\mathcal{C}_{d_1,d_2})_i) z^i,
$$ 
 is called the Poincar\'e series   of the algebras of   joint invariants and covariants.   The finite generation of the algebra of   covariants implies  that its Poincar\'e series is an expansion  of some  rational function.  We consider here the problem of
computing efficiently this rational function.

  The  Poincar\'e series  calculations were an  important  object of research  in  classical invariant theory of the 19th century.
For  the cases  $d\leq 10,$ $d=12$ the  Poincar\'e of the algebra of  invariants of binary form of degree $d$  were calculated by Sylvester and  Franklin, see  \cite{SF}, \cite{Sylv-12}.
Relatively recently, Springer \cite{SP} set the explicit formula for computing the Poincar\'e  series of the algebras of invariants of the binary $d$-forms. This formula has been used by Brouwer and Cohen  \cite{BC} for  $d\leq 17$ and also by Littelmann and Procesi  \cite{LP} for even  $d\leq 36.$
In  \cite{Dr_G} the Poincar\'e series of algebras of   joint invariants and  covariants of two and three binary form of small degrees  is  calculated.

 In the paper we  have  found  Sylvester-Cayley-type formulas  for calculating of $(\mathcal{I}_{d_1,d_2})_i,$ $(\mathcal{C}_{d_1,d_2})_i$    and Springer-type formulas  for calculation of $\mathcal{PI}_{d_1,d_2}(z),$ $ \mathcal{PI}_{d_1,d_2}(z).$ By using the  derived formulas, the Poincar\'e series  is calculated  for $d_1,d_2 \leq 20.$


{\bf 2.} To begin with, we give a proof of the Sylvester-Cayley-type  formula for joint invariants and covariants of two    binary form.
Let  $V\cong \mathbb{C}^2$ be standard  two-dimensional representation of Lie algebra  $\mathfrak{sl_{2}}.$ The irreducible representation   $V_d=\langle v_0,v_1,...,v_d \rangle,$ $\dim V_d=d+1$ of the  algebra $\mathfrak{sl_{2}}$ is the symmetric $d$-power of the standard representation  $V=V_1,$  i.e. $V_d=S^d(V),$  $V_0 \cong  \mathbb{C}.$  
The basis elements    $ \left( \begin{array}{ll}  0\, 1 \\ 0\,0 \end{array} \right),$ $ \left( \begin{array}{ll}  0\, 0 \\ 1\,0 \end{array} \right)$, $ \left( \begin{array}{ll}  1 &  \phantom{-}0 \\  0 &-1 \end{array} \right)$ of the algebra    $\mathfrak{sl_{2}}$ act on    $V_d$  by the derivations  $D_1, D_2, E$ : 
$$
D_1(v_i)=i\, v_{i-1},  D_2(v_i)=(d-i)\,v_{i+1}, E(v_i)=(d-2\,i)\,v_i.
$$
Let us  consider the two irreducible
 $\mathfrak{sl_{2}}$-modules  $V_{d_1}$ і $V_{d_2}.$ Identity the algebras  of polynomial functions  $\mathcal{O}(V_{d_1}),$  $\mathcal{O}(V_{d_2})$ with the symmetrical algebras  $S(V_{d_1}),$ $S(V_{d_2}).$The action of   $\mathfrak{sl_{2}}$  is extended to action on the symmetrical algebra  $S(V_{d_1}\oplus V_{d_2}).$ in the natural way.
The algebra   $\mathcal{I}_{d_1,d_2},$
$$
\mathcal{I}_{d_1,d_2}= \displaystyle{ S(V_{d_1}\oplus V_{d_2})^{\mathfrak{sl_{2}}}}=\{ v \in S(V_{d_1}\oplus V_{d_2})|  D_1(v)=0,  D_2(v) =0 \},
$$
is called the algebra of   joint invariants of two binary forms of degrees   $d_1, d_2.$

Let   $\mathfrak{u}_{2}$ be  the maximal unipotent subalgebra of $\mathfrak{sl}_{2}.$ The algebra  $\mathcal{S}_{d_1,d_2},$
$$
\mathcal{S}_{d_1,d_2}:= \displaystyle{ S(V_{d_1}\oplus V_{d_2})^{\mathfrak{u_{2}}}}=\{ v \in S(V_{d_1}\oplus V_{d_2})|  D_1(v)=0 \},
$$
is called the algebra of   joint semi-invariants of the binary forms of the degrees   $d_1, d_2.$ For any element $v \in \mathcal{S}_{d_1,d_2}$ a natural number $s$ is called the  order of the element $v$ if the number $s$ is the smallest natural number such that \begin{equation*}D_2^s(v) \ne 0, D_2^{s+1}(v) = 0.\end{equation*}
It is  clear that any semi-invariant   of  order $i$ is the highest weight vector  for an  irreducible $\mathfrak{sl_{2}}$-module   of the dimension $i+1$ in $S(V_{d_1}\oplus V_{d_2}).$
 
The application of the Grosshans principle, see  \cite{Gross}, \cite{Pom}  gives
 $$
\mathcal{S}_{d_1,d_2}:= \displaystyle{ S(V_{d_1}\oplus V_{d_2})^{\mathfrak{u_{2}}}}\cong  \mathcal{O}(V_{d_1}\oplus V_{d_2} \oplus \mathbb{C}^2)^{\mathfrak{sl_{2}}}=\mathcal{C}_{d_1,d_2}.
$$
Thus,  the algebra joint covariants is isomorphic to the algebra of  joint semi-invariants. Therefore, it is  enough to compute the Poincar\'e  series of the algebra $\mathcal{S}_{d_1,d_2}.$

The algebra    $S(V_{d_1}\oplus V_{d_2})$ is graded:  
$$
S(V_{d_1}\oplus V_{d_2})=S^0(V_{d_1}\oplus V_{d_2})+S^1(V_{d_1}\oplus V_{d_2})+\cdots +S^n(V_{d_1}\oplus V_{d_2})+\cdots,
$$
and each   $S^n(V_{d_1}\oplus V_{d_2})$ is the complete reducibly
 representation of the Lie algebra  $\mathfrak{sl_{2}}.$

Thus, the following decomposition  holds
$$
S^n(V_{d_1}\oplus V_{d_2}) \cong \gamma_n(d_1,d_2;0) V_0+\gamma_n(d_1,d_2;1) V_1+ \cdots +\gamma_n(d_1,d_2;n\cdot \max(d_1,d_2)) V_{n\cdot \max(d_1,d_2)},  \eqno{(1)}
$$
here  $\gamma_n(d_1,d_2;k)$ is  the  multiplicity of the representation  $V_k$  in the decomposition of  $S^n(V_{d_1}\oplus V_{d_2}).$ On the other hand, the multiplicity  $\gamma_n(d_1,d_2;i)$  of the  representation  $V_i$ is  equal to the number of linearly independent homogeneous joint semi-invariants of  the degree $n$   and the order $i.$
In particular, the number of linearly  independent joint invariants of degree  $n$ is  equal to  $\gamma_n(d_1,d_2;0).$  This argument proves 
\begin{lm}
$$
\begin{array}{ll}
(i) & \dim (\mathcal{I}_{d_1,d_2})_n=\gamma_n(d_1,d_2;0),\\
(ii) & \dim (S_{d_1,d_2})_n=\gamma_n(d_1,d_2;0)+\gamma_n(d_1,d_2;1) + \cdots +\gamma_n(d_1,d_2;n\,\max(d_1,d_2)).
\end{array}
$$
\end{lm}
Let us recall some points of the representation theory of the Lie algebra  $\mathfrak{sl_{2}}.$

The set of weights ( eigenvalues of the operator $E$) of a representation  $W$ denote by  $\Lambda_{W},$  in particular, $\Lambda_{V_d}=\{-d, -d+2, \ldots, d \}.$ The set of weights of a representation  $W$ denote by  $\Lambda_{W},$  in particular,  $\Lambda_{V_d}=\{-d, -d+2, \ldots, d \}.$ 
It is clear that any joint semi-invariant   $v \in \mathcal{S}_{d_1,d_2}$  of degree  $i$ is  the highest weight vector  for an  irreducible representation   $V_i$ in the symmetrical algebra   $S(V_{d_1}\oplus V_{d_2}).$
 A  formal sum 
$$
{\rm Char}(W)=\sum_{k \in \Lambda_{W}} n_W(k) q^k,
$$
is called the character   of a representation  $W,$  
here   $n_W(k)$ denotes the   multiplicity  of the weight $k \in \Lambda_{W}.$
Since, a   multiplicity of any weight of the irreducible representation $V_d$  is  equal to 1, we have  
$$
{\rm Char}(V_d)=q^{-d}+q^{-d+2}+\cdots+q^{d}.
$$

On the other hand, the characted  $ {\rm Char}(S^n(V_{d_1}\oplus V_{d_2})$ of the representation  $S^n(V_{d_1}\oplus V_{d_2})$  equals   $$H_n(q^{-d_1},q^{-d_1+2},\ldots,q^{d_1},q^{-d_2},q^{-d_2+2},\ldots,q^{d_2}),$$ see  \cite{FH},  where  $H_n(x_0,x_1,\ldots,x_{d_1},y_0,y_1,\ldots,y_{d_2})$ is  the complete symmetrical function     $$H_n(x_0,x_1,\ldots,x_{d_1},y_0,y_1,\ldots,y_{d_2})=\sum_{|\alpha|+|\beta|=n} x_0^{\alpha_0} x_1^{\alpha_1} \ldots x_{d_1}^{\alpha_{d_1}}y_0^{\beta_0} y_1^{\beta_1} \ldots y_{d_1}^{\beta_{d_2}} , |\alpha|=\sum_i \alpha_i.$$

By replacing   $x_i=q^{d_1-2\,i},$ $i=0,\ldots, d_1,$ $y_i=q^{d_2-2\,i},$ $i=0,\ldots, d_2,$  we  obtain the specialized expression for the character of  ${\rm Char}(S^n(V_{d_1}\oplus V_{d_2})):$ 
$$
{\rm Char}(S^n(V_{d_1}\oplus V_{d_2}))= \sum_{|\alpha|+|\beta|=n} (q^{d_1})^{\alpha_0} (q^{d_1-2\cdot 1})^{\alpha_1} \ldots (q^{-d_1})^{\alpha_{d_1}}  (q^{d_2})^{\beta_0} (q^{d_2-2\cdot 1})^{\beta_1} \ldots (q^{-d_2})^{\beta_{d_2}} =
$$
$$
= \sum_{|\alpha|+|\beta|=n} p^{d_1|\alpha|+d_2|\beta| +(\alpha_1+2\alpha_2+\cdots + d_1\, \alpha_{d_1})+(\beta_1+2\beta_2+\cdots + d_2\, \beta_{d_2})}=\sum_{i=0}^{n\,\max(d_1,d_2)} \omega_n(d_1,d_2;i) p^{i},
$$
here   $\omega_n(d_1,d_2;i)$  is the number nonnegative integer solutions of the following  system of equations:
$$
\left \{
\begin{array}{l}
d_1|\alpha|+d_2|\beta| -(\alpha_1+2\alpha_2+\cdots + d_1\, \alpha_{d_1})-(\beta_1+2\beta_2+\cdots + d_2\, \beta_{d_2})=i \\
\\
|\alpha|+|\beta| =n.
\end{array}
\right. 
 \eqno{(2)}
$$

We can summarize what we have shown so far in  
\begin{te} 
$$
\begin{array}{ll}
(i) & \dim (\mathcal{I}_{d_1,d_2})_n=\omega_n(d_1,d_2;0)-\omega_n(d_1,d_2;2),\\
&\\
(ii) & \dim (\mathcal{S}_{d_1,d_2})_n=\omega_n(d_1,d_2;0)+\omega_n(d_1,d_2;1).
\end{array}
$$

\end{te}
\begin{proof}
\noindent
 $(i)$
The zero weight appears  once  in any  representation $V_k,$  for even $k$,  therefore
$$
\omega_n(d_1,d_2;0)=\gamma_n(d_1,d_2;0)+\gamma_n(d_1,d_2;2)  +\gamma_n(d_1,d_2;4)+\ldots
$$
The weight   $2$ appears  once in any  representation   $V_k,$  for odd $k$,  therefore
$$
\omega_n(d_1,d_2;0)=\gamma_n(d_1,d_2;2)+\gamma_n(d_1,d_2;4)  +\gamma_n(d_1,d_2;6)+\ldots
$$
Taking into account Lemma  1,  obtain
$$
\gamma_n(d_1,d_2;0)= \dim (\mathcal{I}_{d_1,d_2})_n=\omega_n(d_1,d_2;0)-\omega_n(d_1,d_2;2).
$$

\noindent
$(ii)$
The  weight $1$ appears  once  in any  representation $V_k,$  for even $k$,  therefore
 $$
\omega_n(d_1,d_2;1)=\gamma_n(d_1,d_2;1)+\gamma_n(d_1,d_2;3)  +\gamma_n(d_1,d_2;5)+\ldots
$$
Thus,
$$
\begin{array}{l}
\displaystyle \omega_n(d_1,d_2;0)+\omega_n(d_1,d_2;1)=\\ \\
\displaystyle =\gamma_n(d_1,d_2;0)+\gamma_n(d_1,d_2;1)  +\gamma_n(d_1,d_2;2)+\ldots
+\gamma_n(d_1,d_2;n\,\max(d_1,d_2))=\\\\
\displaystyle =\dim (S_{d_1,d_2})_n.
\end{array}
$$
\end{proof}

\noindent
{\bf 3.} Simplify the system  $(2)$  to 
$$
\left \{
\begin{array}{l}
d_1\alpha_0+(d_1-2)\alpha_1+(d_1-4)\alpha_2+\cdots + (-d_1)\, \alpha_{d_1}+\\+d_2\beta_0+(d_2-2)\beta_1+(d_2-4)\beta_2+\cdots + (-d_2)\, \beta_{d_2}=i \\
\\
\alpha_0+\alpha_1+\cdots +\alpha_{d_1}+ \beta_0+\beta_1+\cdots +\beta_{d_2}=n.
\end{array}
\right. 
$$
 It well-known  that  the number
 $\omega_n(d_1,d_2;i)$ of   non-negative integer solutions of the above  system
is equal to the coefficient of  $\displaystyle t^n z^i $ of the  expansion   of the  series
$$
f_{d_1,d_2}(t,z)=\frac{1}{(1-t z^{d_1})(1-t\,z^{d_1-2})\ldots (1-t\,z^{-d_1})(1-t z^{d_2})(1-t\,z^{d_2-2})\ldots (1-t\,z^{-d_2})}.
$$
Denote it in such a way:  $\omega_n(d_1,d_2;i):=\left[  t^n z^i\right](f_{d_1,d_2}(t,z)).$  
The  following statement holds
\begin{te} 
$$
\begin{array}{ll}
(i) & \dim (I_{d_1,d_2})_n=[t^n ](1-z^2)f_{d_1,d_2}(t,z),\\
&\\
(ii) & \dim (\mathcal{S}_{d_1,d_2})_n=[t^n ](1+z)f_{d_1,d_2}(t,z).
\end{array}
$$
\end{te}
\begin{proof}
Taking into account the formal property $[x^{i-k}]f(x)=[x^{i}](x^k f(x)),$ we  get
$$
\begin{array}{l}
 \dim (I_{d_1,d_2})_n=\omega_n(d_1,d_2;0)-\omega_n(d_1,d_2;2)=[t^n]f_{d_1,d_2}(t,z)-[t^n\,z]f_{d_1,d_2}(t,z)=\\
\\
=[t^n]f_{d_1,d_2}(t,z)-[t^n]z f_{d_1,d_2}(t,z)=[t^n](1-z^2) f_{d_1,d_2}(t,z).
\end{array}
$$
In the same way
$$
\begin{array}{l}
 \dim (S_{d_1,d_2})_n=\omega_n(d_1,d_2;0)+\omega_n(d_1,d_2;1)=[t^n]f_{d_1,d_2}(t,z)-[t^n\,z]f_{d_1,d_2}(t,z)=\\
\\
=[t^n]f_{d_1,d_2}(t,z)+[t^n]z f_{d_1,d_2}(t,z)=[t^n](1+z) f_{d_1,d_2}(t,z).
\end{array}
$$
\end{proof}

{\bf 4.} Let us prove  Springer-type  formula for the Poincar\'e  series  $\mathcal{PI}_{d_1,d_2}(z),$ $ \mathcal{PC}_{d_1,d_2}(z)$  of the  algebras  joint invariants and  covariants of  the two binary forms.
Consider the $\mathbb{C}$-algebra $\mathbb{C}[[t,z]]$   of  formal   power series.
For arbitrary   $m,n \in \mathbb{Z^+}$  define  $\mathbb{C}$-linear function
$$ \Psi_{m,n}:\mathbb{C}[[t,z]] \to \mathbb{C}[[z]],$$ $ m,n  \in \mathbb{Z}^{+} $ in the  following  way:
$$
\Psi_{m,n}\left(\sum_{i,j=0}^{\infty} a_{i,j}t^i z^j\right)=\sum_{i=0}^{\infty} a_{i m,i n} z^i.
$$
Define by $\varphi_n$  the restriction of 
$\Psi_{m,n}$ to  $\mathbb{C}[[z]],$ namely 

$$
\varphi_{n}\left(\sum_{i=0}^{\infty} a_{i}z^i \right)=\sum_{i=0}^{\infty} a_{i n} z^i. 
$$
There is  an effective algorithm of   calculation  for the function $ \varphi_n, $ see  
 \cite{B-PS}.
In  some  cases   calculation of the functions $ \Psi$  can  be reduced   to  calculation of  the functions    $\varphi$.  The following  statements hold:

\begin{lm} For  $R(z) \in \mathbb{C}[[z]]$    and  for  $ m, n, k  \in \mathbb{N}$    we have:

$$
\begin{array}{ll}
(i) & \displaystyle  \Psi_{1,n}\left( \frac{R(z)}{1-t z^k} \right)=\left \{ \begin{array}{l}  \varphi_{n-k}(R(z)), n\geq k, \\  \\ 0,  \text{ if  } k>n, \end{array} \right. \\
\\

(ii) &  \displaystyle  \Psi_{1,n}\left( \frac{R(z)}{(1-t z^k)^2} \right)=\left \{ \begin{array}{l} (z\, \varphi_{n-k}(R(z)))', n\geq k, \\  \\ 0,  \text{  if   } k>n. \end{array} \right. 
\end{array}
$$
\end{lm}
\begin{proof}

\noindent
$(i)$
 Let  $R(z)=\sum_{j=0}^{\infty} r_{j} z^j.$  Then for   $k < n$  we  have  
$$
\begin{array}{l}
\displaystyle \Psi_{1,n}\left( \frac{R(z)}{1-t z^k} \right)=\Psi_{1,n}\Big( \sum_{j,s\geq 0} r_j  z^j (t z^k)^s\Big)
{=}\Psi_{1,n}\Big(\sum_{s\geq 0} r_{s(n-k)} (t z^n)^s \Big){=}\sum_{s \geq 0}  r_{s(n-k)} z^s.
\end{array}
$$
On other hand, $\displaystyle \varphi_{n-k}(R(z))=\varphi_{n-k}\Bigl(\sum_{j=0}^{\infty} r_{j} z^j\Bigr){=\sum_{s \geq 0}  r_{s(n-k)} z^s.}$

\noindent
$(ii)$
Let  $R(z)=\sum_{j=0}^{\infty} r_{j} z^j.$  Observe,  that 
$$
\frac{1}{(1-x)^2}=\left(\frac{1}{1-x} \right)'=1+2x+3x^2+\ldots 
$$
Then for   $n>k$  we  have  
$$
\begin{array}{l}
\displaystyle \Psi_{1,n}\left( \frac{R(z)}{(1-t z^k)^2} \right)=\Psi_{1,n}\Big( \sum_{j,s\geq 0} (s+1)\,r_j z^j (t z^k)^s\Big)=\\ \\ 
\displaystyle =\Psi_{1,n}\Big(\sum_{s\geq 0} (s+1) r_{s(n-k)}\, (t z^n)^s \Big){=}\sum_{s \geq 0}  (s+1)\, r_{s(n-k)} z^s.
\end{array}
$$
On other hand,
 $$
\left(z \varphi_{n-k}(R(z))\right)'=\left(\sum_{s=0}^{\infty} r_{s(n-k)} z^{s+1}\right)'=\sum_{s \geq 0} (s+1) r_{s(n-k)} z^s.
$$
 \end{proof}

The main idea of this calculations is that 
the  Poincar\'e series $ \mathcal{PI}_{d_1,d_2}(z),$ $\mathcal{PS}_{d_1,d_2}(z)$  can be expressed  in terms of functions $ \Psi.$ The following simple but important statement  holds

\begin{lm}  Let   $d:=\max(d_1,d_2).$ Then 
$$
\begin{array}{ll}
(i) & \mathcal{PI}_{d_1,d_2}(z)=\Psi_{1,d}\left((1-z^2)f_{d_1,d_2}(tz^d,z)\right),\\
\\
(ii) & \mathcal{PS}_{d_1,d_2}(z)=\Psi_{1,d}\left((1+z)f_{d_1,d_2}(tz^d,z)\right),\\

\end{array}
$$

\end{lm}
\begin{proof} Theorem 2   implies  that  $\dim\left((\mathcal{I}_{d_1,d_1})_n\right)=[t ^n](1-z^2)f_{d_1,d_2}(t,z).$ 
Then 
$$
\begin{array}{l}
\displaystyle  \mathcal{PI}_{d_1,d_2}(z) = \sum_{n=0}^{\infty}  \dim(I_{d_1,d_1})_n z^n=\sum_{n=0}^{\infty} \bigl([t ^n](1-z^2)f_{d_1,d_2}(t,z)\bigr)z^n{=}\\
\\
\displaystyle =\sum_{n=0}^{\infty} \bigl([(tz^d) ^n](1-z^2)f_{d_1,d_2}(tz^d,z)\bigr)z^n{=} \Psi_{1,d}\left((1-z^2)f_{d_1,d_2}(tz^d,z)\right).
\end{array}
$$
Similarly, we  get   the  statement $(ii).$ 

We  replace  $t$  with   $tz^d$  to avoid  of a negative powers  of $z$  in the denominator of the function $f_{d_1,d_2}(t,z).$
\end{proof}
Write  the function $f_{d_1,d_2}(tz^{d_2},z)$ in the following way 
$$
f_{d_1,d_2}(tz^{d_2},z)=\frac{1}{(tz^{d_2-d_1},z^2)_{d_1+1} (t,z^2)_{d_2+1}},
$$
here $(a,q)_n=(1-a) (1-a\,q)\cdots (1-a\,q^{n-1}).$

By using a  representation of the function  $\Psi_{1,d}$ via the  contour integral ,  see \cite{B-PS},  we get  two   new  formulas  for the  Puancar\'e series:
$$
\displaystyle  \mathcal{PI}_{d_1,d_2}(t) {=}\frac{1}{2\pi i} \oint_{|z|=1} \frac{ 1-z^2}{(tz^{d_2-d_1},z^2)_{d_1+1} (t,z^2)_{d_2+1}} \frac{dz}{z},
$$
$$
\displaystyle  \mathcal{PS}_{d_1,d_2}(t) {=}\frac{1}{2\pi i} \oint_{|z|=1} \frac{ 1+z}{(tz^{d_2-d_1},z^2)_{d_1+1} (t,z^2)_{d_2+1}} \frac{dz}{z},  d_2 \geq d_1.
$$
Compare the  formula  with the Molien-Weyl integral formula  for the Poincar\'e series  of the algebra of  invariants of   binary form, see \cite{DerK}, p. 183.

Now we  can present    Springer-type  formula  for the Poincar\'e  series $P_{d}(z)$ $\mathcal{PI}_{d_1,d_2}(z)$ i $\mathcal{PS}_{d_1,d_2}(z).$
\begin{te}  Let $d_2-d_2=1 \pmod 2$  and  $d_2> d_1.$  Then
$$
\begin{array}{l}
\displaystyle \mathcal{PI}_{d_1,d_2}(z) {=}
 \sum_{d_1/2 \leq k \leq d_1}\varphi_{2k-d_1}\left( (1-z^2) A_k(z)\right)+
\displaystyle \sum_{ k=0}^{[d_2/2]}\varphi_{d_2-2k}\left( (1-z^2)B_k(z)\right),\end{array}
$$
where
$$
A_k(z)=\frac{(-1)^{k} z^{(d_1-k)(d_1-k+1)+(d_2+1)(d_1-2k)}}{(z^2,z^2)_{k} (z^2,z^2)_{d_1-k}(z^{(d_1+d_2)-2k},z^2)_{d_2+1}},
$$
$$
B_k(z)=
\left\{
\begin{array}{l}
\displaystyle \frac{(-1)^k z^{k(k+1)}}{(z^{d_2-d_1-2k},z^2)_{d_1+1}(z^2,z^2)_k(z^2,z^2)_{d_2-k}} \text{  for } 2k<d_2-d_1,\\
\\
\displaystyle  \frac{(-1)^{\frac{d_2-d_1+1}{2}} z^{k(k+2)-1/2(d_2-d_1+1)}}{(z,z^2)_{s+1}(z,z^2)_{d_1-s}(z,z^2)_k(z^2,z^2)_{d_2-k}} \text{  for } s=\frac{2k-(d_2-d_1)-1}{2}.
\end{array}
\right.
$$
\end{te}
\begin{proof} If the integer number  $d_1,$ $d_2,$ $d_2>d_1$  have   different parity then in the denominator of the function  $f_{d_1,d_2}(tz^{d_2},z)$ all factors appears in the degree 1. Then the rational function  $f_{d_1,d_2}(tz^{d_2},z)$ has  the following partial fractions decomposition
$$
f_{d_1,d_2}(tz^{d_2},z)=\sum_{k=0}^{d_1} \frac{A_k(z)}{1-t z^{d_2+d_1-2k}}+\sum_{k=0}^{d_2} \frac{B_k(z)}{1-t z^{2k}}.
$$
By direct calculations we get 
$$
\begin{array}{l}
\displaystyle  A_k(z)=\lim_{t \to z^{2k-(d_1+d_2)}}\left(f_{d_1,d_2}(tz^{d_2},z)(1-t z^{d_2+d_1-2k}) \right)=\\
\\
\displaystyle =\frac{1}{(1-z^{2k})(1-z^{2k-2})\cdots (1-z^{2})(1-z^{-2}) \cdots (1-z^{-2(d_1-k)})(z^{2k-(d_1+d_2)},z^2)_{d_2+1}}=\\
\\
\displaystyle =\frac{(-1)^{d_1-k}z^{2+4+\ldots +2(d_1-k)}(-1)^{d_2+1}z^{(d_2+1)(d_1+d_2-2k)-2(1+2+\ldots+d_2)}}{(z^2,z^2)_k(1-z^2) \cdots (1-z^{2(d_1-k)}) (z^{(d_1+d_2)-2k},z^2)_{d_2+1}}=\\
\\
\displaystyle  =\frac{(-1)^{k} z^{(d_1-k)(d_1-k+1)+(d_2+1)(d_1-2k)}}{(z^2,z^2)_{k} (z^2,z^2)_{d_1-k}(z^{(d_1+d_2)-2k},z^2)_{d_2+1}}.
\end{array}
$$

In the same  way, we obtain 
$$
B_k(z)=\lim_{t \to z^{-2k}}\left(f_{d_1,d_2}(tz^{d_2},z)(1-t z^{2k}) \right)=
$$
$$
=\left\{
\begin{array}{l}
\displaystyle \frac{(-1)^k z^{k(k+1)}}{(z^{d_2-d_1-2k},z^2)_{d_1+1}(z^2,z^2)_k(z^2,z^2)_{d_2-k}} \text{  for } 2k<d_2-d_1,\\
\\
\displaystyle  \frac{(-1)^{\frac{d_2-d_1+1}{2}} z^{k(k+2)-1/2(d_2-d_1+1)}}{(z,z^2)_{s+1}(z,z^2)_{d_1-s}(z,z^2)_k(z^2,z^2)_{d_2-k}} \text{  for  } 2k=d_2-d_1+2s+1.

\end{array}
\right.
$$
Using  the above lemmas we obtain
$$
\begin{array}{l}
\displaystyle \mathcal{PI}_{d_1,d_2}(z) {=}\Psi_{1,d_2}\bigl((1-z^2)f_d(tz^d_2,z)\bigr){=}\Psi_{1,d_2}\left( \sum_{k=0}^{d_1} \frac{(1-z^2)A_k(z)}{1-t z^{d_2+d_1-2k}}+\sum_{k=0}^{d_2} \frac{(1-z^2)B_k(z)}{1-t z^{2k}} \right){=}\\
\displaystyle
= \sum_{k=0}^{d_1}\Psi_{1,d_2}\left(  \frac{(1-z^2)A_k(z)}{1-t z^{d_2+d_1-2k}}\right)+ \sum_{ k =0}^{d_2}\Psi_{1,d_2}\left( \frac{(1-z^2)B_k(z)}{1-t z^{2k}}\right).
\end{array}
$$
By substituting the expression of  $A_k(z),$  $B_k(z)$ and, taking into account the Lemma 2, we  get the statement of the theorem. 
\end{proof}

Let us consider the  case $d_2=d_1 \mod 2$  and  $d_2> d_1.$  Then 
$$
\begin{array}{l}
f_{d_1,d_2}(tz^{d_2},z)^{-1}=
\\
\\
\displaystyle={(1-t)\ldots(1-tz^{d_2-d_1-2})(1-tz^{d_2-d_1})^2\ldots(1-tz^{d_2+d_1})^2(1-tz^{d_2+d_1+2})\ldots (1-tz^{2d_2})}=\\
\\
\displaystyle=(t,z^2)_{(d_2-d_1)/2}(tz^{d_2-d_1},z^2)_{d_1+1}^2(tz^{d_2+d_1+2},z^2)_{(d_2-d_1)/2}.
\end{array}
$$
Consider the partial fraction decomposition of  the rational function $f_{d_1,d_2}(tz^{d_2},z):$
$$
\begin{array}{l}
\displaystyle f_{d_1,d_2}(tz^{d_2},z)=\\
\\
\displaystyle=\sum_{k=0}^{d_1}\left(\frac{A_k(z)}{1-t z^{d_2+d_1-2k}}+ \frac{B_k(z)}{(1-t z^{d_2+d_1-2k})^2}\right)+\sum_{k=0}^{(d_2-d_1)/2-1} \frac{C_k(z)}{1-t z^{2k}}+\sum_{k=(d_2+d_1)/2+1}^{d_2} \frac{C_k(z)}{1-t z^{2k}}.
\end{array}
$$
It is  easy to see that 
$$
\begin{array}{l}
\displaystyle A_k(z)=-\frac{1}{z^{d_2+d_1-2k}}\lim_{t \to z^{2k-(d_1+d_2)}}\left( f_{d_1,d_2}(tz^{d_2},z) (1-t z^{d_2+d_1-2k})^2\right)'_t,\\
\\
\displaystyle B_k(z)=\lim_{t \to z^{2k-(d_1+d_2)}}\left( f_{d_1,d_2}(tz^{d_2},z) (1-t z^{d_2+d_1-2k})^2\right),\\
\\
\displaystyle C_k(z):=\lim_{t \to z^{-2k}}\left( f_{d_1,d_2}(tz^{d_2},z) (1-t z^{2k})\right).
\end{array}
$$
Thus,
\begin{te} For   $d_1=d_2 \pmod 2$ and $d_2 > d_1$  the Poincar\'e  series $\mathcal{PI}_{d_1,d_2}(z)$ is calculated  by the formulas
$$
\begin{array}{l}
\mathcal{PI}_{d_1,d_2}(z) {=}\Psi_{1,d_2}((1-z^2)f_{d_1,d_2}(tz^{d_2},z))=\\
\\
=\displaystyle\sum_{d_1/2\leq k\leq d_1}\varphi_{2k-d_1}\left((1-z^2)A_k(z)\right)+\sum_{d_1/2\leq k\leq d_1} \left( z \, \varphi_{2k-d_1} \left((1-z^2)B_k(z)\right) \right)'_z+\\
\\
+\displaystyle \sum_{0\leq 2k\leq d_2-d_1-2}\varphi_{d_2-2k}\left((1-z^2)C_k(z)\right)+ \sum_{d_1+d_2+2\leq 2k\leq 2d_2}\varphi_{d_2-2k}\left((1-z^2)C_k(z)\right).

\end{array}
$$
\end{te}

For the case  $d_2=d_1=d$  we  have
\begin{te} Let $d_2=d_1=d.$ 
Then 
$$
\mathcal{PI}_{d,d}(z)=\displaystyle\sum_{0\leq k\leq d/2}\varphi_{d-2k}\left((1-z^2)A_k(z)\right)+\sum_{0\leq k\leq d/2}\left( z\, \varphi_{d-2k}\left( (1-z^2)B_k(z)\right)\right)'_z,
$$
where,
$$
\begin{array}{l}
\displaystyle A_k(z)=-\frac{1}{z^{2k}}\lim_{t \to z^{-2k}}\left( f_{d,d}(tz^{d},z) (1-t z^{2k})^2\right)'_t,\\
\\
\displaystyle B_k(z)=\lim_{t \to z^{d-2k}}\left( f_{d,d}(tz^{d},z) (1-t z^{2k})^2\right).\\
\end{array}
$$

\end{te}
By replacing the factor $1-z^2$  with $1+z$ in $\mathcal{PI}_{d_1,d_2}(z)$  we get the Poincar\'e series  $\mathcal{PS}_{d_1,d_2}(z)$  

{\bf 5.}  For direct computations of the functions $\varphi$ we use the following technical lemma, see \cite{B-PS}:

\begin{lm} Let   $R(z)$ be some polynomial of $z.$ Then
$$
\varphi_n\left(\frac{R(z)}{(1-z^{k_1})(1-z^{k_2})\cdots(1-z^{k_m})} \right)= 
\frac{\varphi_n\bigr(R(z)Q_n(z^{k_1})Q_n(z^{k_2})Q_n(z^{k_m})\bigr)}{(1-z^{k_1})(1-z^{k_2})\cdots(1-z^{k_m})},
$$
here  $Q_n(z)=1+z+z^2+\ldots+z^{n-1},$  and  $k_i$ are natural numbers.
\end{lm}

{\bf Example.}  Consider  the case $d_1=1,d_2=3.$  We  have
$$
\begin{array}{l}
\displaystyle \mathcal{PI}_{1,3}(z)=\Psi_{1,3}\left({\frac {1-{z}^{2}}{ \left( 1-t{z}^{2} \right) ^{2} \left( 1-t{z}^{4}
 \right) ^{2} \left( 1-t \right)  \left( 1-t{z}^{6} \right) }}
 \right)=
\\
\\
=\displaystyle \Psi_{1,3}\left(\frac{(1-z^2)\,A_1(z)}{(1-tz^2)} \right) + \Psi_{1,3}\left(\frac{(1-z^2)\,B_1(z)}{(1-tz^2)^2} \right)+ \Psi_{1,3}\left(\frac{(1-z^2)\,C_0(z)}{(1-t)} \right).
\end{array}
$$
Since, $3=1 \mod 2$,  using Teorema  4.   We have 
$$
A_1(z)={\frac {{z}^{2} \left( 3\,{z}^{4}+{z}^{2}-1 \right) }{ \left(1- {z}^{4}
 \right) ^{2} \left( 1-{z}^{2} \right) ^{3}}},B_1(z)={\frac {{z}^{2}}{ \left(1-{z}^{2} \right) ^{3} \left( 1-{z}^{4}
 \right) }}, C_0(z)={\frac {1}{ \left( 1-{z}^{2} \right) ^{2} \left( 1-{z}^{4}\right) ^
{2} \left(1-{z}^{6}\right) }}.
$$
Therefore
$$
\begin{array}{l}
\displaystyle \mathcal{PI}_{1,3}(z)=\\
\displaystyle{=}\varphi_1\left({\frac {{z}^{2} \left( 3\,{z}^{4}+{z}^{2}-1 \right) }{ \left(1- {z}^{4}
 \right) ^{2} \left( 1-{z}^{2} \right) ^{2}}}\right){+} \left( z\,\varphi_1\left({\frac {{z}^{2}}{ \left(1-{z}^{2} \right) ^{2} \left( 1-{z}^{4} \right) }}\right)\right)'{+}\varphi_3\left({\frac {1}{ \left( 1-{z}^{2} \right)  \left( 1-{z}^{4}\right) ^{2} \left(1-{z}^{6}\right) }}\right).
\end{array}
$$
Taking into account Lemma 4  and $\varphi_1(F(z))=F(z),$    obtain
$$
\begin{array}{l}
\displaystyle \mathcal{PI}_{1,3}(z)=\\ \\
\displaystyle{=}{\frac {{z}^{2} \left( 3\,{z}^{4}+{z}^{2}-1 \right) }{ \left(1- {z}^{4}
 \right) ^{2} \left( 1-{z}^{2} \right) ^{2}}}{+} \left( {\frac {{z}^{3}}{ \left(1-{z}^{2} \right) ^{2} \left( 1-{z}^{4} \right) }}\right)'{+} \frac{\displaystyle \varphi_3 \left( \left(\frac { 1-{z}^{6}}{1-z^2}\right) \left( \frac{ 1-{z}^{12}}{1-z^4} \right)^{2}\right)}{(1-z^2)^2 (1-z^4)^2}=\\ \\
\displaystyle ={\frac {{z}^{2} \left( 3\,{z}^{4}+{z}^{2}-1 \right) }{ \left(1- {z}^{4}
 \right) ^{2} \left( 1-{z}^{2} \right) ^{2}}}-{\frac { \left( 5\,{z}^{4}+4\,{z}^{2}+3 \right) {z}^{2}}{ \left( 1-{
z}^{2} \right) ^{4} \left( 1+{z}^{2}\right) ^{2}}}+\\ \\ \displaystyle + \frac{\varphi_3\left( {z}^{20}+{z}^{18}+3\,{z}^{16}+2\,{z}^{14}+5\,{z}^{12}+3\,{z}^
{10}+5\,{z}^{8}+2\,{z}^{6}+3\,{z}^{4}+{z}^{2}+1 \right)
}{(1-z^2)^2 (1-z^4)^2}=\\ \\
\displaystyle  =-{\frac {{z}^{2} \left( 2\,{z}^{4}+3\,{z}^{2}+4 \right) }{ \left( 1-{
z}^{2} \right) ^{3} \left( 1+{z}^{2}\right)  \left( {1-z}^{4} \right) }}+ \frac{{z}^{6}+5\,{z}^{4}+2\,{z}^{2}+1 
}{(1-z^2)^2 (1-z^4)^2}={\frac {{z}^{4}-{z}^{2}+1}{ \left(1-{z}^{2} \right)  \left(1-{z}^{4}
 \right) ^{2}}}=\\ \\
=\displaystyle {\frac {(1+z^2)({z}^{4}-{z}^{2}+1)(1+z^4)}{ (1+z^2)\left(1-{z}^{2} \right)  \left(1-{z}^{4}
 \right) ^{2}(1+z^4)}}=\frac {{z}^{10}+{z}^{6}+{z}^{4}+1}{ \left( 1-{z}^{4} \right) ^{2}
 \left( 1-{z}^{8} \right)}.
\end{array}
$$
By using Lemma  3 the Poncar\'e series of the algebras of   joined  invariants and covariants for the 
 $d_1,d_2 \leq 20$  is found. Belos are results for the cases  $d_1,d_2 \leq 5.$  Note that $\mathcal{PC}_{d_1,d_2}(z)=\mathcal{PS}_{d_1,d_2}(z).$
$$
\mathcal{PI}_{1,1}(z)=\frac{1}{ 1-{z}^{2}},\mathcal{PI}_{1,2}(z)=\frac{1}{\left( 1-{z}^{2} \right)  \left( 1-{z}^{3} \right)},
$$
$$
\mathcal{PI}_{1,3}(z)=\frac {{z}^{10}+{z}^{6}+{z}^{4}+1}{ \left( 1-{z}^{4} \right) ^{2}
 \left( 1-{z}^{8} \right)},
\mathcal{PI}_{1,4}(z)={\frac {{z}^{13}+{z}^{11}+{z}^{9}+{z}^{4}+{z}^{2}+1}{ \left( 1-{z}^{3}
 \right)  \left( 1-{z}^{5} \right)  \left( 1-{z}^{6} \right) ^{2}}},
$$
$$
\mathcal{PI}_{1,5}(z)={\frac {{z}^{26}+2\,{z}^{20}+6\,{z}^{18}+3\,{z}^{16}+7\,{z}^{14}+7\,{z
}^{12}+3\,{z}^{10}+6\,{z}^{8}+2\,{z}^{6}+1}{ \left( 1-{z}^{4} \right) 
^{2} \left( 1-{z}^{6} \right)  \left( 1-{z}^{8} \right)  \left( 1-{z}^
{12} \right) }},
$$
$$
\mathcal{PI}_{2,2}(z)=\frac{1}{( 1-{z}^{2})^3}, \mathcal{PI}_{2,3}(z)={\frac {{z}^{9}+{z}^{7}+{z}^{2}+1}{ \left( 1-{z}^{3} \right)  \left( 1
-{z}^{4} \right) ^{2} \left( 1-{z}^{5} \right) }},
$$
$$
 \mathcal{PI}_{2,4}(z)={\frac {{z}^{6}+1}{ \left( 1-{z}^{2} \right) ^{2} \left( 1-{z}^{3}
\right) ^{2} \left( 1-{z}^{4} \right) }},
$$
$$
\mathcal{PI}_{2,5}(z)=\frac{pi_{2,5}(z)}{\left( 1-{z}^{3} \right)  \left( 1-{z}^{4} \right)  \left( 1-{z}^{5}
 \right)  \left( 1-{z}^{6} \right)  \left( 1-{z}^{7} \right)  \left( 1
-{z}^{8} \right)},
$$
$$
\begin{array}{l}
pi_{2,5}(z)={z}^{24}+{z}^{22}+{z}^{20}+{z}^{18}+2\,{z}^{17}+3\,{z}^{16}+5\,{z}^{15
}+5\,{z}^{14}+8\,{z}^{13}+7\,{z}^{12}+8\,{z}^{11}+\\+5\,{z}^{10}+5\,{z}^{
9}+3\,{z}^{8}+2\,{z}^{7}+{z}^{6}+{z}^{4}+{z}^{2}+1,
\end{array}
$$
$$
\mathcal{PI}_{3,3}(z)={\frac {{z}^{8}-{z}^{6}+2\,{z}^{4}-{z}^{2}+1}{ \left( 1-{z}^{2}
 \right) ^{2} \left( 1-{z}^{4} \right) ^{3}}},
$$
$$
\mathcal{PI}_{3,4}(z)={\frac {pi_{3,4}}{ \left( 1-{z}^{4} \right) ^{2} \left( 1-{z}^{3} \right) 
 \left( 1-{z}^{6} \right)  \left( 1-{z}^{5} \right)  \left( 1-{z}^{7}
 \right) }},
$$
$$
\begin{array}{l}
pi_{3,4}(z)={z}^{20}+{z}^{18}+{z}^{15}+{z}^{14}+3\,{z}^{13}+4\,{z}^{12}+6
\,{z}^{11}+6\,{z}^{10}+6\,{z}^{9}+4\,{z}^{8}+3\,{z}^{7}+{z}^{6}+\\+{z}^{5
}+{z}^{2}+1,
\end{array}
$$
$$
\mathcal{PI}_{3,5}(z)=\frac {pi_{3,5}}{\left( 1-{z}^{4} \right) ^{2} \left( 1-{z}^{6} \right) ^{2} \left( 1-{z}^{8} \right)^{3}},
$$
$$
\begin{array}{l}
pi_{3,5}(z)={z}^{34}+4\,{z}^{30}+5\,{z}^{28}+22\,{z}^{26}+34\,{z}^{24}+65\,{z}^{22
}+77\,{z}^{20}+94\,{z}^{18}+94\,{z}^{16}+77\,{z}^{14}+\\+65\,{z}^{12}+34
\,{z}^{10}+22\,{z}^{8}+5\,{z}^{6}+4\,{z}^{4}+1,
\end{array}
$$
$$
\mathcal{PI}_{4,5}(z)=\frac {pi_{4,5}}{\left( 1-{z}^{3} \right)  \left( 1-{z}^{4} \right) ^{2} \left( 1-{z}^
{5} \right)  \left( 1-{z}^{6} \right)  \left( 1-{z}^{7} \right) 
 \left( 1-{z}^{8} \right)  \left( 1-{z}^{9} \right)
},
$$
$$
\begin{array}{l}
pi_{4,5}(z)={z}^{35}+{z}^{33}+2\,{z}^{30}+4\,{z}^{29}+8\,{z}^{28}+13\,{z}^{27}+21
\,{z}^{26}+27\,{z}^{25}+38\,{z}^{24}+47\,{z}^{23}+\\+54\,{z}^{22}+62\,{z}
^{21}+68\,{z}^{20}+70\,{z}^{19}+74\,{z}^{18}+74\,{z}^{17}+70\,{z}^{16}
+68\,{z}^{15}+62\,{z}^{14}+54\,{z}^{13}+\\+47\,{z}^{12}+38\,{z}^{11}+27\,
{z}^{10}+21\,{z}^{9}+13\,{z}^{8}+8\,{z}^{7}+4\,{z}^{6}+2\,{z}^{5}+{z}^
{2}+1,
\end{array}
$$

$$
\mathcal{PI}_{4,4}(z)=\frac{pi_{4,4}(z)}{ \left( 1-{z}^{2} \right)  \left( 1-{z}^{3} \right) ^{3} \left( 1-{z}^
{4} \right)  \left( 1-{z}^{6} \right)  \left( 1-{z}^{8} \right) 
}
$$
$$
\begin{array}{l}
pi_{4,4}(z)={z}^{19}+2\,{z}^{17}+{z}^{16}+3\,{z}^{15}+2\,{z}^{14}+4\,{z}^{13}+3\,{
z}^{12}+4\,{z}^{11}+4\,{z}^{10}+4\,{z}^{9}+4\,{z}^{8}+3\,{z}^{7}+\\+4\,{z
}^{6}+2\,{z}^{5}+3\,{z}^{4}+{z}^{3}+2\,{z}^{2}+1,
\end{array}
$$
$$
\mathcal{PI}_{5,5}(z)=\frac{pi_{5,5}(z)}{ \left( 1-{z}^{2} \right)  \left( 1-{z}^{4} \right) ^{2} \left( 1-{z}^
{6} \right) ^{2} \left( 1-{z}^{8} \right) ^{3} \left( 1-{z}^{12}
 \right)
}
$$
$$
\begin{array}{l}
pi_{5,5}(z)={z}^{46}+4\,{z}^{42}+5\,{z}^{40}+44\,{z}^{38}+74\,{z}^{36}+188\,{z}^{
34}+259\,{z}^{32}+452\,{z}^{30}+575\,{z}^{28}+\\+723\,{z}^{26}+773\,{z}^{
24}+773\,{z}^{22}+723\,{z}^{20}+575\,{z}^{18}+452\,{z}^{16}+259\,{z}^{
14}+188\,{z}^{12}+74\,{z}^{10}+\\+44\,{z}^{8}+5\,{z}^{6}+4\,{z}^{4}+1,
\end{array}
$$

$$
\mathcal{PC}_{1,1}(z)=\frac{1}{(1-z)^2(1-z^2)}, \mathcal{PC}_{1,2}(z)={\displaystyle \frac {z^{2} + 1}{(1 - z)^{2}\,(1
 - z^{2})\,(1 - z^{3})}} ,
$$

$$
 \mathcal{PC}_{1,3}(z)={\displaystyle \frac {z^{6} + z^{4} + 3\,z^{3} + 
z^{2} + 1}{( 1- z )^{2}\,(1 - z^{2})\,( 1- z^{4} )^{2}}},  \mathcal{PC}_{1,4}(z)={\displaystyle \frac {z^{8} + 2\,z^{6} + 2\,z^{5}
 + 4\,z^{4} + 2\,z^{3} + 2\,z^{2} + 1}{(1 - z^{5})\,(1 - z)^{2}\,
(1 - z^{3})^{2}\,(1 - z^{2})}} 
$$
$$
\mathcal{PC}_{1,5}(z)=\frac{pc_{1,5}(z)}{ \left( 1-z \right) ^{2} \left( 1-{z}^{4} \right) ^{2}  \left( 1-{z}^{6} \right) ^{2} \left( 1-{z}^{8
} \right) 
 },
$$
$$
\begin{array}{l}
pc_{1,5}(z)={z}^{22}+3\,{z}^{20}+6\,{z}^{19}+10\,{z}^{18}+18\,{z}^{17}+24\,{z}^{16
}+34\,{z}^{15}+43\,{z}^{14}+44\,{z}^{13}+57\,{z}^{12}+ \\+53\,{z}^{11}+57
\,{z}^{10}+44\,{z}^{9}+43\,{z}^{8}+34\,{z}^{7}+24\,{z}^{6}+18\,{z}^{5}
+10\,{z}^{4}+6\,{z}^{3}+3\,{z}^{2}+1,
\end{array}
$$
$$
\mathcal{PC}_{2,2}(z)=\frac{1+z^2}{(1-z)^2(1-z^2)^3},\mathcal{PC}_{2,3} :={\frac {{z}^{9}+3\,{z}^{7}+3\,{z}^{6}+4\,{z}^{5}+4\,{z}^{4}+3\,{z}^{3}
+3\,{z}^{2}+1}{\left( 1-z \right) ^{2}\left( 1-z^2 \right)  
 \left(1-{z}^{3}\right)  \left(1-{z}^{4} \right) \left( 1-{z}^{5} \right) }},
$$
$$
\mathcal{PC}_{2,4}(z)={\frac {{z}^{6}+2\,{z}^{4}+4\,{z}^{3}+2\,{z}^{2}+1}{ \left( 1-z
 \right) ^{2} \left( -{z}^{2}+1 \right) ^{3} \left( -{z}^{3}+1
 \right) ^{2}}},
$$
$$
\mathcal{PC}_{2,5}=\frac{pc_{2,5}}{\left( 1-z \right) ^{2} \left( -{z}^{3}+1 \right)  \left( 1-{z}^{4}
 \right)  \left( 1-{z}^{5} \right)  \left( 1-{z}^{6} \right)  \left( 1
-{z}^{7} \right)  \left( 1-{z}^{8} \right)
 },
$$
$$
\begin{array}{l}
pc_{2,5}(z)={z}^{26}+5\,{z}^{24}+8\,{z}^{23}+19\,{z}^{22}+31\,{z}^{21}+52\,{z}^{20
}+76\,{z}^{19}+104\,{z}^{18}+135\,{z}^{17}+162\,{z}^{16}+\\+188\,{z}^{15}
+200\,{z}^{14}+209\,{z}^{13}+200\,{z}^{12}+188\,{z}^{11}+162\,{z}^{10}
+135\,{z}^{9}+104\,{z}^{8}+76\,{z}^{7}+52\,{z}^{6}+\\+31\,{z}^{5}+19\,{z}
^{4}+8\,{z}^{3}+5\,{z}^{2}+1,

\end{array}
$$
$$
\mathcal{PC}_{3,3}:={\frac {{z}^{10}+3\,{z}^{8}+6\,{z}^{7}+6\,{z}^{6}+6\,{z}^{5}+6\,{z}^{4
}+6\,{z}^{3}+3\,{z}^{2}+1}{ \left( 1-z \right) ^{2} \left(1-{z}^{2}
 \right) ^{2} \left( 1-{z}^{4} \right) ^{3}}},
$$
$$
\mathcal{PC}_{3,4}:= \frac{pc_{3,4}}{ \left( 1-z \right) ^{2} \left(1 -{z}^{3}\right)  \left( 1-{z}^{4}
 \right) ^{2} \left( 1-{z}^{5} \right)  \left( 1-{z}^{6} \right) 
 \left( 1-{z}^{7} \right)
 },
$$
$$
\begin{array}{l}
pc_{3,4}(z)={z}^{22}+6\,{z}^{20}+8\,{z}^{19}+21\,{z}^{18}+34\,{z}^{17}+52\,{z}^{16
}+76\,{z}^{15}+95\,{z}^{14}+117\,{z}^{13}+127\,{z}^{12}+\\+134\,{z}^{11}+
127\,{z}^{10}+117\,{z}^{9}+95\,{z}^{8}+76\,{z}^{7}+52\,{z}^{6}+34\,{z}
^{5}+21\,{z}^{4}+8\,{z}^{3}+6\,{z}^{2}+1
\end{array}
$$
$$
\mathcal{PC}_{4,4}:= \frac{pc_{4,4}}{\left( 1-z \right) ^{2} \left( 1-{z}^{2} \right)  \left( 1-{z}^{3}
 \right) ^{2} \left( 1-{z}^{4} \right) ^{2} \left( 1-{z}^{6} \right) 
 \left( 1-{z}^{8} \right)
 },
$$
$$
\begin{array}{l}
pc_{4,4}(z)={z}^{22}+7\,{z}^{20}+10\,{z}^{19}+24\,{z}^{18}+42\,{z}^{17}+62\,{z}^{
16}+88\,{z}^{15}+113\,{z}^{14}+134\,{z}^{13}+145\,{z}^{12}+\\+156\,{z}^{
11}+145\,{z}^{10}+134\,{z}^{9}+113\,{z}^{8}+88\,{z}^{7}+62\,{z}^{6}+42
\,{z}^{5}+24\,{z}^{4}+10\,{z}^{3}+7\,{z}^{2}+1
\end{array}
$$
$$
\mathcal{PC}_{4,5}:= \frac{pc_{4,5}}{\left( 1-z \right) ^{2} \left( 1-{z}^{3} \right)  \left( 1-{z}^{4}
 \right) ^{2} \left( 1-{z}^{5} \right)  \left( 1-{z}^{6} \right) 
 \left( 1-{z}^{7} \right)  \left( 1-{z}^{8} \right)  \left( 1-{z}^{9} \right)
 },
$$
$$
\begin{array}{l}
pc_{4,5}(z)={z}^{37}+8\,{z}^{35}+15\,{z}^{34}+45\,{z}^{33}+93\,{z}^{32}+181\,{z}^{
31}+324\,{z}^{30}+531\,{z}^{29}+828\,{z}^{28}+\\+1202\,{z}^{27}+1674\,{z}
^{26}+2206\,{z}^{25}+2789\,{z}^{24}+3377\,{z}^{23}+3929\,{z}^{22}+4392
\,{z}^{21}+4734\,{z}^{20}+\\+4909\,{z}^{19}+4909\,{z}^{18}+4734\,{z}^{17}
+4392\,{z}^{16}+3929\,{z}^{15}+3377\,{z}^{14}+2789\,{z}^{13}+2206\,{z}
^{12}+\\+1674\,{z}^{11}+1202\,{z}^{10}+828\,{z}^{9}+531\,{z}^{8}+324\,{z}
^{7}+181\,{z}^{6}+93\,{z}^{5}+45\,{z}^{4}+15\,{z}^{3}+8\,{z}^{2}+1
\end{array}
$$
$$
\mathcal{PC}_{5,5}:= \frac{pc_{5,5}}{ \left( 1-z \right) ^{2} \left( 1-{z}^{2} \right)  \left( 1-{z}^{4}
 \right) ^{2} \left( 1-{z}^{6} \right) ^{3} \left( 1-{z}^{8} \right) ^
{3}
 },
$$
$$
\begin{array}{l}
pc_{5,5}(z)={z}^{42}+8\,{z}^{40}+20\,{z}^{39}+56\,{z}^{38}+126\,{z}^{37}+257\,{z}^
{36}+506\,{z}^{35}+891\,{z}^{34}+1438\,{z}^{33}+\\+2332\,{z}^{32}+3380\,{
z}^{31}+4939\,{z}^{30}+6488\,{z}^{29}+8707\,{z}^{28}+10720\,{z}^{27}+
13175\,{z}^{26}+15010\,{z}^{25}+\\+17283\,{z}^{24}+18414\,{z}^{23}+19791
\,{z}^{22}+19578\,{z}^{21}+19791\,{z}^{20}+18414\,{z}^{19}+17283\,{z}^
{18}+\\+15010\,{z}^{17}+13175\,{z}^{16}+10720\,{z}^{15}+8707\,{z}^{14}+
6488\,{z}^{13}+4939\,{z}^{12}+3380\,{z}^{11}+2332\,{z}^{10}+\\+1438\,{z}^
{9}+891\,{z}^{8}+506\,{z}^{7}+257\,{z}^{6}+126\,{z}^{5}+56\,{z}^{4}+20
\,{z}^{3}+8\,{z}^{2}+1
\end{array}
$$


\begin{thebibliography}{15}


\bibitem{SF}
Sylvester, J. J.,  Franklin, F., Tables of the generating functions and groundforms for the binary quantic of the first ten orders, Am. J. II., 223--251, 1879
\bibitem{Sylv-12}
Sylvester, J. J.
Tables of the generating functions and groundforms of the binary duodecimic, with some general remarks, and tables of the irreductible syzigies of certain quantics. Am. J. IV. 41-62, 1881.
\bibitem{SP}
Springer T.A., On the invariant theory of SU(2), Indag. Math. 42 (1980), 339-345.

\bibitem{BC}    Brouwer A.,  Cohen A., The Poincare series of the polynomial invariants under
$SU_2$  in its irreducible representation of degree $\leq  17$, preprint of the Mathematisch Centrum,
Amsterdam, 1979.

\bibitem{LP}
 Littelmann P.,  Procesi C., On the Poincar\'e series of the invariants of binary forms, J. Algebra 133, No.2, 490-499 (1990).

\bibitem{Dr_G} Drensky, V,  Genov, G.K.
Multiplicities of Schur functions with applications to invariant theory and PI-algebras.[J] C. R. Acad. Bulg. Sci. 57, No. 3, 5-10, 2004.
\bibitem{Gross} Grosshans F., Observable groups and Hilbert's fourteenth problem.
Amer. J. Math. 95 (1973), 229-253.

\bibitem{Pom} Pommerening, K., 
Invariants of unipotent groups. - A survey.  In  Invariant theory, Symp. West Chester/Pa. 1985, Lect. Notes Math. 1278, 8-17, 1987.

\bibitem{FH} W.Fulton,  J. Harris, Reptesentation theory: a first course, 1991.
\bibitem{DerK}
 Derksen H.,  Kemper G., Computational Invariant Theory, Springer-Verlag, New York,
2002.

\bibitem{B-PS} L. Bedratyuk The Poincar\'e  series of the covariants of binary forms. arXiv:0904.1325

\end{thebibliography}
\end{document}